\documentclass[11pt]{amsart}
\usepackage[utf8]{inputenc}
\usepackage[T1]{fontenc}
\usepackage{amsmath}
\usepackage{amssymb}
\usepackage{eulerpx,bm}
\usepackage{comment}
\usepackage{longtable}
\usepackage{xcolor,soul} 
\usepackage{colortbl}
\usepackage{hyperref}
\newtheorem{thm}{Theorem}[section]
\newtheorem{prop}[thm]{Proposition}
\newtheorem{lem}[thm]{Lemma}
\newtheorem{cor}[thm]{Corollary}

\newtheorem*{obs*}{Observation}

\newtheorem{innercustomthm}{Theorem}
\newenvironment{customthm}[1]
{\renewcommand\theinnercustomthm{#1}\innercustomthm}
{\endinnercustomthm}

\theoremstyle{definition}

\newtheorem{exmp}[thm]{Example}

\theoremstyle{remark}

\newtheorem{rem}[thm]{Remark}

\numberwithin{equation}{section}

\newcommand{\R}{\mathbb{R}}  % The real numbers.
\newcommand{\C}{\mathbb{C}}  % The complex numbers.
\newcommand{\Z}{\mathbb{Z}}  % The integers.
\newcommand{\N}{\mathbb{N}}  % The natural numbers.
\newcommand{\PP}{\mathbb{P}}  % Projective space.
\counterwithout{equation}{section}
\begin{document}

\title[A Refinement of Hilbert's 1888 Theorem]{A Refinement of Hilbert's 1888 Theorem: Separating Cones along the Veronese Variety}
\author{Charu Goel}
\address{Department of Basic Sciences, Indian Institute of Information Technology, Nagpur, India}
\email{charugoel@iiitn.ac.in}
%\urladdr{www.math.sc.edu/$\sim$howard} % Delete if not wanted.
%%
%% If there is another author uncomment and edit the following.
%%
\author{Sarah Hess}
\address{Department of Mathematics and Statistics, University of Konstanz, Konstanz, Germany}
\email{Sarah.Hess@uni-konstanz.de}
%\urladdr{www.math.sc.edu/$\sim$second}
%%
%% If there are three of more authors they are added in the obvious
%% way. 
\author{Salma Kuhlmann}
\address{Department of Mathematics and Statistics, University of Konstanz, Konstanz, Germany}
\email{Salma.Kuhlmann@uni-konstanz.de}
%\footnote[1]{Draft 6, \today}
%%%
%%% The following is for the abstract.  The abstract is optional and
%%% if not used just delete, or comment out, the following.
%%%
\begin{abstract}
	For $n,d\in\N$, the cone $\mathcal{P}_{n+1,2d}$ of positive semi-definite (PSD) $(n+1)$-ary $2d$-ic forms (i.e., homogeneous polynomials with real coefficients in $n+1$ variables of degree $2d$) contains the cone $\Sigma_{n+1,2d}$ of those that are representable as finite sums of squares (SOS) of $(n+1)$-ary $d$-ic forms. Hilbert's 1888 Theorem \cite{Hilbert} states that $\Sigma_{n+1,2d}=\mathcal{P}_{n+1,2d}$ exactly in the \textit{Hilbert cases} $(n+1,2d)$ with $n+1=2$ or $2d=2$ or $(3,4)$. For the \textit{non-Hilbert cases}, we examine in \cite{GHK}, a specific cone filtration
	\begin{equation} \label{abs:eq1} \Sigma_{n+1,2d}=C_0\subseteq \ldots \subseteq C_n \subseteq C_{n+1} \subseteq \ldots \subseteq C_{k(n,d)-n}=\mathcal{P}_{n+1,2d}\end{equation} along $k(n,d)+1-n$ projective varieties containing the Veronese variety via the Gram matrix method \cite{CLR}. Here, $k(n,d)+1$ is the dimension of the real vector space of $(n+1)$-ary $d$-ic forms. In particular, we compute the number $\mu(n,d)$ of \textit{strictly separating intermediate cones} (i.e., $C_i$ such that $\Sigma_{n+1,2d}\subsetneq C_i \subsetneq \mathcal{P}_{n+1,2d}$) for the cases $(3,6)$ and $(n+1,2d)_{n\geq 3,d=2,3}$.
	
	In this paper, firstly, we generalize our findings from \cite{GHK} to any non-Hilbert case by identifying each strict inclusion in (\ref{abs:eq1}). This allows us to give a refinement of Hilbert's 1888 Theorem by computing $\mu(n,d)$. The cone filtration from (\ref{abs:eq1}) thus reduces to a specific cone subfiltration \begin{equation}
    \Sigma_{n+1,2d}=C_0^\prime\subsetneq C_1^\prime \subsetneq \ldots \subsetneq C_{\mu(n,d)}^\prime \subsetneq C_{\mu(n,d)+1}^\prime=\mathcal{P}_{n+1,2d} \end{equation} in which each inclusion is strict. Secondly, we show that each $C_i^\prime$, and hence each strictly separating $C_i$, fails to be a spectrahedral shadow. 
\end{abstract}
%%
%%  LaTeX will not make the title for the paper unless told to do so.
%%  This is done by uncommenting the following.
%%
\maketitle
%%
%% LaTeX can automatically make a table of contents.  This is done by
%% uncommenting the following:
%%
%\tableofcontents
%%%%%%%%%%%%%%%%%%%%%%%%%%%%%%%%%%%%%%%%%%%%%%%%%%%%%%%%%%%%%%%%%%%%%%%%%%%%%%%%%%%%%%%%%%%%%%%%%%%%%%%%%%%%%%%%%%%%%%%%%%%%%%%%%%%%%
%%%%%%%%%%%%%%%%%%%%%%%%%%%%%%%%%%%%%%%%%%%%%%%%%%%%%%%%%%%%%%%%%%%%%%%%%%%%%%%%%%%%%%%%%%%%%%%%%%%%%%%%%%%%%%%%%%%%%%%%%%%%%%%%%%%%%
\section{Introduction}
\subsection{General Overview}
It is crucial for both real algebraic geometry and polynomial optimization to be able to determine if a given positive semi-definite (PSD) $(n+1)$-ary $2d$-ic (i.e, a non-negative form in $n+1$ variables of degree $2d$) can be represented by a finite sum of squares (SOS) of half degree $d$ forms. In 1888, Hilbert \cite{Hilbert} classified all cases $(n+1,2d)$ for which any PSD $(n+1)$-ary $2d$-ic admits a SOS-representation. These are the \textit{Hilbert cases} $(2,2d)$, $(n+1,2)$ and $(3,4)$. In all other cases, the \textit{non-Hilbert cases}, Hilbert non-constructively verified the existence of PSD-not-SOS $(n+1)$-ary $2d$-ics by a reduction to the \textit{basic non-Hilbert cases} $(4,4)$ and $(3,6)$ via an argument which allows to increase the number of variables and the even degree considered while preserving the PSD-not-SOS property. In any non-Hilbert case, the cone $\Sigma_{n+1,2d}$ of SOS $(n+1)$-ary $2d$-ics is thus a proper subcone of the cone $\mathcal{P}_{n+1,2d}$ of PSD $(n+1)$-ary $2d$-ics. It was approximately 80 years after Hilbert's seminal paper that the first explicit examples of a PSD-not-SOS ternary sextic and a PSD-not-SOS quaternary quartic were given by Motzkin \cite{Motzkin} and Robinson \cite{Rob}, respectively. Soon thereafter, Choi and Lam \cite{CL2, CL} introduced examples of a PSD-not-SOS ternary sextic and a PSD-not-SOS quaternary quartic.
Over the years a significant interest in geometric properties of the cones $\Sigma_{n+1,2d}$ and $\mathcal{P}_{n+1,2d}$ has been developed. For instance, both $\Sigma_{n+1,2d}$ and $\mathcal{P}_{n+1,2d}$ are convex semialgebraic sets. Moreover, $\Sigma_{n+1,2d}$ is a spectrahedral shadow (i.e., the image of a spectrahedron under an affine linear map). In fact, it is known that any spectrahedral shadow is a convex semialgebraic set. In 2006, Nemirovski \cite{Nemirovski} asked if the converse is true as well. In 2009, Helton and Nie \cite{HeltonNie} conjectured that any convex semialgebraic set (over $\mathbb{R}$) is a spectrahedral shadow. During the next decade several results supporting the Helton-Nie conjecture in special cases were given, e.g., \cite{HN2,NetzerSanyal,Scheiderer2}. In 2018, Scheiderer \cite{Scheiderer} disproved the Helton-Nie conjecture by developing a method to produce convex semialgebraic sets that are not spectrahedral shadows. In particular, he gave the first counterexample by showing that $\mathcal{P}_{n+1,2d}$ is not a spectrahedral shadow in the non-Hilbert cases. It is thus interesting to investigate at what stage between the cones $\Sigma_{n+1,2d}$ and $\mathcal{P}_{n+1,2d}$ the property of being a spectrahedral shadow is lost. \\

\enlargethispage{\baselineskip}
\noindent In \cite{GHK}, we constructed and examined a specific cone filtration 
\begin{equation}\label{int:eq1} \Sigma_{n+1,2d}=C_0 \subseteq \ldots \subseteq C_n \subseteq C_{n+1} \subseteq \ldots\subseteq C_{k(n,d)-n}=\mathcal{P}_{n+1,2d}\end{equation} along $k(n,d)-n+1$ projective varieties containing the Veronese variety via the Gram matrix method introduced in \cite{CLR}. Here, $k(n,d)+1$ is the dimension of the real vector space of $(n+1)$-ary $d$-ics. For $n\geq 3$, we established that $C_0=\ldots=C_n$. For $n=2$, we even showed $C_0=C_1=C_2=C_3$. Moreover, if $d=2$ or $3$, we proved the remaining inclusions in (\ref{int:eq1}) being strict, i.e., 
$$\begin{cases} \Sigma_{n+1,2d}=C_{n}\subsetneq \ldots \subsetneq C_{k(n,d)-n}=\mathcal{P}_{n+1,2d},& \mathrm{\ for\ } (n+1,2d)_{n\geq 3,d=2,3} \\ \Sigma_{3,6}=C_{2+1}\subsetneq \ldots \subsetneq C_{9-2}=\mathcal{P}_{3,6}, & \mathrm{\ for\ } (n+1,2d)=(3,6). \end{cases}$$
More precisely, we showed that the number $\mu(n,d)$ of \textit{strictly separating intermediate cones} (i.e., $C_i$ such that $\Sigma_{n+1,2d}\subsetneq C_i \subsetneq \mathcal{P}_{n+1,2d}$) in (\ref{int:eq1}) in these particular cases is given by $$\mu(n,d)=\begin{cases}(k(n,d)-n)-(n+1), & \mbox{for } (n+1,2d)_{n\geq 3, d=2,3} \\3, & \mbox{for } (n+1,2d)=(3,6). \end{cases}$$ Here, note that for a given $f\in\mathcal{P}_{n+1,2d}\backslash\Sigma_{n+1,2d}$, there exists a unique $i(f)\in\{0,\ldots,k(n,d)-n-1\}$ such that $f\in C_{i(f)+1}\backslash C_{i(f)}$. Furthermore, $\vert \{i(f) \mid f\in F\}\vert \leq \mu(n,d)+1$ for any $F\subseteq\mathcal{P}_{n+1,2d}\backslash\Sigma_{n+1,2d}$. In fact, there exists some $F\subseteq\mathcal{P}_{n+1,2d}\backslash\Sigma_{n+1,2d}$ such that $\vert \{i(f) \mid f\in F\}\vert = \mu(n,d)+1$ which we call a \textit{complete set of separating forms}. We computed $\mu(n,d)$ for $d=2,3$ by producing complete sets of separating forms using the Motzkin ternary sextic, the Choi-Lam ternary sextic and the Choi-Lam quaternary quartic. \\

\noindent In this paper, we now firstly generalize our findings from \cite{GHK} to arbitrary non-Hilbert cases by identifying each strict inclusion in (\ref{int:eq1}). This gives us a refinement of Hilbert's 1888 Theorem by explicitly computing $\mu(n,d)$ and reduces (\ref{int:eq1}) to the cone subfiltration \begin{equation} \label{int:eq2} \Sigma_{n+1,2d}=C_0^\prime\subsetneq C_1^\prime \subsetneq \ldots \subsetneq C_{\mu(n,d)}^\prime \subsetneq C_{\mu(n,d)+1}^\prime=\mathcal{P}_{n+1,2d}
\end{equation} in which each inclusion is strict. Secondly, we apply the methods from \cite{Scheiderer} to show that, for $i=1,\ldots,\mu(n,d)$, each $C_i^\prime$ in (\ref{int:eq2}), and hence each strictly separating $C_i$ in (\ref{int:eq1}), fails to be a spectrahedral shadow. Consequently, we provide many explicit examples of separating cones between $\Sigma_{n+1,2d}$ and $\mathcal{P}_{n+1,2d}$ which are convex semialgebraic sets but not spectrahedral shadows. In particular, the property of being a spectrahedral shadow is lost as soon as we move from $\Sigma_{n+1,2d}$ to the first strictly separating $C_i$ in (\ref{int:eq1}).

\subsection{Preliminaries}
For $n,\, l\in\N$, we denote the vector space of real forms in $n+1$ variables of degree $l$ by $\mathcal{F}_{n+1,l}$ and call $f\in\mathcal{F}_{n+1,l}$ a \textit{$(n+1)$-ary $l$-ic}. For $K\subseteq\R^{n+1}$, we call $f$
\textit{locally positive semi-definite on $K$}, if $f(x)\geq 0$ for all $x\in K$. In the special case of 
$K=\R^{n+1}$, $f$ is said to be \textit{(globally) positive semi-definite}. For $d\in\N$, the set $\mathcal{P}_{n+1,2d}$ of all PSD forms in $\mathcal{F}_{n+1,2d}$ is a closed pointed full-dimensional convex cone (see \cite{Reznick}). We further call $f\in\mathcal{F}_{n+1,2d}$ a \textit{sum of squares}, if $f=\sum\limits_{i=1}^sg_i^2$ for some $g_1,\ldots,g_s\in\mathcal{F}_{n+1,d}$ and $s\in\N$.
The set $\Sigma_{n+1,2d}$ of all SOS forms in $\mathcal{F}_{n+1,2d}$ is also a closed pointed full-dimensional convex cone (see \cite{Reznick}). Clearly, $\Sigma_{n+1,2d}\subseteq\mathcal{P}_{n+1,2d}$ and we set $\Delta_{n+1,2d}:=\mathcal{P}_{n+1,2d}\backslash\Sigma_{n+1,2d}$. We say that $f\in\mathcal{P}_{n+1,2d}$ is \textit{PSD-extremal} if $f=f_1+f_2$ for $f_1,f_2\in\mathcal{P}_{n+1,2d}$ implies $f_i=\lambda_i f$ for some $\lambda_i\geq 0$ ($i\in\{1,2\}$). If $f$ is PSD-extremal, then, for $j=0,\ldots,n$, also the form $X_j^2f$ is PSD-extremal  (see \cite[Theorem 5.1 (iii)]{CKLR}).

\subsubsection*{\textbf{Circuit Forms}}
Set $\N_0:=\N\cup\{0\}$ and let $f\in\mathcal{F}_{n+1,2d}$ be supported on $A\subseteq\N_0^{n+1}$. Assume that all elements in the set $\mbox{vert}(A)$ of vertices of the convex hull of $A$ are even.  If $f(X)=\sum\limits_{j=0}^r f_{a(j)}X^{a(j)}+f_bX^{b}$ for $r\leq n+1$, pairwise distinct $a(0),\ldots,a(r),b\in A$, $f_{a(0)},\ldots,f_{a(r)}>0$, $f_b\in\R$ such that
\begin{enumerate}
    \item $vert(A)=\{a(0),\ldots,a(r)\}$,
    \item $a(0),\ldots,a(r)$ are affinely independent and
    \item there exist unique $\lambda_0,\ldots,\lambda_r>0$ with $\sum\limits_{j=0}^r \lambda_j =1$ and $b=\sum\limits_{j=0}^r\lambda_j a(j)$,
\end{enumerate} then $f$ is called a \textit{circuit form} with \textit{outer exponents} $a(0),\ldots,a(r)$. 
We denote the set of circuit forms in $\mathcal{F}_{n+1,2d}$ by $\mathfrak{C}_{n+1,2d}$. Furthermore, we also set $\mathcal{P}^\mathfrak{C}_{n+1,2d}:=\mathcal{P}_{n+1,2d}\cap\mathfrak{C}_{n+1,2d}$ and $\Delta^\mathfrak{C}_{n+1,2d}:=\Delta_{n+1,2d}\cap\mathfrak{C}_{n+1,2d}$. 

We order the set $\{\alpha\in\N_0^{n+1}\mid\vert\alpha\vert=d\}$ lexicographically which gives us the ordered set $\{\alpha_0,\ldots,\alpha_{k(n,d)}\}$ for $k(n,d):=\dim(\mathcal{F}_{n+1,d})-1=\binom{n+d}{n}-1$. This allows us to enumerate the monomial basis of $\mathcal{F}_{n+1,d}$ along the exponents $\alpha_j=(\alpha_{j,0},\ldots,\alpha_{j,n})$ by setting $m_j(X):=X^{\alpha_j}$ for $j=0,\ldots,k(n,d)$. For later purpose, for $i=0,\ldots,k(n,d)$, we set $\mathfrak{S}_i$ to be the $\R$-span of $m_sm_t$ for $0\leq s,t \leq i$. For $f\in\mathfrak{C}_{n+1,2d}$, the outer exponents are thus given by $2\alpha_{j_0},\ldots,2\alpha_{j_r}$ for some $j_0,\ldots,j_r\in\{0,\ldots,k(n,d)\}$. For future purpose, we set $j(f):=\max\{j_0,\ldots,j_r\}$.

\subsubsection*{\textbf{Construction of a Cone Filtration between $\Sigma_{n+1,2d}$ and $\mathcal{P}_{n+1,2d}$}} \label{Gram}
The following approach is based on the Gram matrix method \cite{CLR} and was introduced in detail in \cite{GHK}. We thus only briefly recall the relevant notations:

For $n,d\in\N$, let $k(n,d):=\dim(\mathcal{F}_{n+1,d})-1$ and write $k$ instead of $k(n,d)$ whenever $n,d$ are clear from the context. Set $\mathrm{Sym}_{k+1}(\R)$ to be the vector space of $(k+1)\times(k+1)$ real symmetric matrices, $Z:=(Z_0,\ldots,Z_k)$ and  \begin{eqnarray} \label{Q} Q\colon\mathrm{Sym}_{k+1}(\R) &\to&\mathcal{F}_{k+1,2} \\
	\nonumber			A 							&\mapsto&q_A(Z):=ZAZ^t.
\end{eqnarray} We consider the surjective \textit{Gram map}
\begin{eqnarray*} \mathcal{G} \colon \mathrm{Sym}_{k+1}(\R) & \to & \mathcal{F}_{n+1,2d} \\
									A				& \mapsto & f_A(X):=q_A(m_0(X)\ldots m_k(X))
\end{eqnarray*} and, for $f\in\mathcal{F}_{n+1,2d}$, call $A\in\mathcal{G}^{-1}(f)$ a \textit{Gram matrix associated to $f$} with corresponding \textit{quadratic form} $q_A$ \textit{(associated to $f$)}.
The map $\mathcal{G}$ is in general not injective. Hence, the kernel of $\mathcal{G}$ is non-trivial and can be characterized via the \textit{(projective) Veronese embedding}
\begin{eqnarray*} V\colon \PP^n & \to & \PP^k \\
							\, [x] & \mapsto & [m_0(x):\ldots:m_k(x)]
\end{eqnarray*} whose image $V(\PP^n)$ is called the \textit{Veronese variety}. For a generic Gram matrix $A_f$ associated to $f\in\mathcal{F}_{n+1,2d}$, we obtain
\begin{eqnarray}
\label{Kern} \mathcal{G}^{-1}(f)&=& \{A\in\mathrm{Sym}_{k+1}(\R)\mid q_{A-A_f} \mbox{\ vanishes\ on\ } V(\PP^n)\}
\end{eqnarray}
Moreover, the following characterization of SOS and PSD forms can be made:
\begin{eqnarray}
	\label{CharSOS2}
	\Sigma_{n+1,2d}
	&=&\{f\in\mathcal{F}_{n+1,2d}\mid \exists\ A\in \mathcal{G}^{-1}(f) \colon q_A\vert_{_{\PP^{k}(\R)}}\geq 0\} \\
    \label{CharPSD} 
	\mathcal{P}_{n+1,2d} &=&\{f\in\mathcal{F}_{n+1,2d}\mid \exists\ A\in \mathcal{G}^{-1}(f) \colon q_A\vert_{_{\mathcal{V}(\PP^n)(\R)}}\geq 0\}.
\end{eqnarray}
For a set $W\subseteq\PP^k$, this motivates the definition of the convex cone $$C_W:=\{f\in\mathcal{F}_{n+1,2d}\mid \exists\ A\in \mathcal{G}^{-1}(f) \colon q_A\vert_{_{W(\R)}}\geq 0\}.$$ Clearly, $\Sigma_{n+1,2d}=C_{\PP^k}$ and $\mathcal{P}_{n+1,2d}=C_{V(\PP^n)}$. Thus, if $V(\PP^n)\subseteq W$, then $C_W$ is an intermediate cone between $\Sigma_{n+1,2d}$ and $\mathcal{P}_{n+1,2d}$. Consequently, $C_W$ is pointed and full-dimensional. \\

We now propose a specific filtration of intermediate cones between $\Sigma_{n+1,2d}$ and $\mathcal{P}_{n+1,2d}$ that is obtained from assigning several varieties containing the Veronese variety to $W$: 
To this end, for $i=0,\ldots,k-n$, we set $$H_i:=\{[z]\in\PP^k\mid\exists\ x\in \C^{n+1}\colon(z_0,\ldots,z_{n+i})=(m_0(x),\ldots,m_{n+i}(x))\}$$ and let $V_i\subseteq \PP^k$ be the Zariski closure of $H_i$ in $\PP^k$ (that is, the smallest projective variety in $\PP^k$ containing $H_i$). This gives us a specific filtration of projective varieties $V(\PP^n)=V_{k-n} \subsetneq \ldots \subsetneq V_0=\PP^k$ with a corresponding specific filtration of sets of real points
\begin{eqnarray} \label{SetOfRealPointsFiltration} V(\PP^n)(\R)=V_{k-n}(\R) \subsetneq \ldots \subsetneq V_0(\R)=\PP^k(\R)
\end{eqnarray} in which all inclusions are strict. For later purpose, we recall from \cite{GHK} explicit descriptions of $V_0,\ldots,V_{k-n}$: For $i=1,\ldots,k-n$, we set $$s_i:=\min\{s\in\{1,\ldots,k\}\mid \exists\ t\in\{s,\ldots,k\}\colon \alpha_{s}+\alpha_{t}=\alpha_0+\alpha_{n+i}\}$$ and let $t_i\in\{s_i,\ldots,k\}$ be such that $\alpha_{s_i}+\alpha_{t_i}=\alpha_0+\alpha_{n+i}$. Next, we set $q_i(Z):=Z_0Z_{n+i}-Z_{s_i}Z_{t_i}$ and recall from \cite[Theorem 2.2]{GHK} that $V_i$ is the projective closure of the affine variety $$K_i:=\mathcal{V}(q_1(1,Z_1,\ldots,Z_k),\ldots,q_i(1,Z_1,\ldots,Z_k))\subseteq\C^k$$ under the embedding 
$$\begin{array}{cccc} \phi\colon&\C^k&\to&\PP^k \\ &(z_1,\ldots,z_k)&\mapsto& [1:z_1:\ldots:z_k].
\end{array}$$

Lastly, for $i=0,\ldots,k-n$, we set $C_i:=C_{V_i}$ and obtain the cone filtration
\begin{equation} \label{ConeFiltration}
	\Sigma_{n+1,2d}=C_0\subseteq \ldots \subseteq C_{k-n}=\mathcal{P}_{n+1,2d}.
\end{equation}
It is clear that at least one inclusion in (\ref{ConeFiltration}) is strict in any non-Hilbert case. However, it is not clear how many among these inclusions are strict and which ones. This brings us to the following question:
\begin{center} ($Q$) \textit{Which inclusions in (\ref{ConeFiltration}) are strict?} \end{center}
In \cite[Theorem A]{GHK}, we showed for $n\geq 3$ that $C_0=\ldots=C_n$. If $n=2$, we even established $C_0=C_1=C_2=C_3$. Question ($Q$) thus reduces to:
\begin{center} ($Q^\prime$) \textit{Which inclusions in $C_{n}\subseteq C_{n+1} \subseteq \ldots \subseteq C_{k-n}$ are strict?} \end{center}
In \cite[Theorem 4.1, 4.2, C and D]{GHK}, we demonstrated that, for the cases $(n+1,4)_{n\geq 3}$ and $(n+1,6)_{n\geq 2}$, $C_{n+1} \subsetneq \ldots \subsetneq C_{k-n}$. In the cases $(n+1,4)_{n\geq 3}$, we moreover established $C_n\subsetneq C_{n+1}$. It thus remains to answer question ($Q^\prime$) for the cases $(n+1,2d)_{n\geq 2,d\geq 4}$.

\subsubsection*{\textbf{Spectrahedral Shadows}} \label{Spec}
For $l\in\N$, let $\mathrm{Sym}_l^+(\R)$ denote the set of all positive semi-definite matrices in $\mathrm{Sym}_{l}(\R)$. Note that $\mathrm{Sym}_l^+(\R)$ is a closed convex cone in $\mathrm{Sym}_l(\R)$. For $m\in\N$, let $K\subseteq\R^m$. If there exist an $l\in\N$ and an affine linear map $T:\R^m\to\mathrm{Sym}_l(\R)$ such that $K$ is the preimage of $\mathrm{Sym}_l^+(\R)$ under $T$, then $K$ is called a \textit{spectrahedron}. However, if there exists an $l\in\N$, a spectrahedron $S\in\R^l$ and a linear map $T\colon \R^l\to\R^m$ such that $K=T(S)$, then $K$ is called a \textit{spectrahedral shadow}. We refer an interested reader to \cite{NetzerPlaumann} for an introduction to spectrahedral shadows and their essential properties.

\subsection{Structure of the Paper and Main Results} \label{Structure}
In Section \ref{Cone:Description}, we prove for any $q\in\mathcal{F}_{k+1,2}$ that if $q\vert_{_{\phi(K_i)(\R)}}\geq 0$, then $q\vert_{_{\overline{\phi(K_i)}(\R)}}\geq 0$ (where the closure is taken w.r.t.\ the Zariski topology). This gives us the first main result:
\begin{customthm}{A} \label{cor:extendingCi} Let $n,d\geq 1$ and $i=0,\ldots,k-n$, then $C_i=C_{\phi(K_i)}$.
\end{customthm}

\noindent Recall that for $f\in\Delta_{n+1,2d}$, there exists a unique $i(f)\in\{1,\ldots,k-n-1\}$ such that $f\in C_{i(f)+1}\backslash C_{i(f)}$. In Section \ref{if}, we firstly use Theorem \ref{cor:extendingCi} to compute, for given $f\in\Delta_{n+1,2d}^{\mathfrak{C}}$, a non-trivial upper bound for $i(f)$. Secondly, we show that this upper bound is sharp if $f$ is moreover PSD-extremal. As a consequence, the second main result, i.e., Theorem \ref{cor:greatestcone} below, follows. Recall that for a circuit form $f$ with outer exponents $2\alpha_{j_0},\ldots,2\alpha_{j_r}$, $j(f):=\max\{j_0,\ldots,j_r\}$.

\begin{customthm}{B} \label{cor:greatestcone} For $n,d\geq 2$, $(n+1,2d)\neq (3,4)$, let $f\in\Delta^\mathfrak{C}_{n+1,2d}$ be PSD-extremal, then
$i(f)=j(f)-n-1$.
\end{customthm}

\noindent In \cite{GHK}, for $n\geq 2$ and $d=3$, we constructed explicit complete sets of separating forms $I_n\subseteq\Delta_{n+1,6}$. In Section \ref{Degree-Jumping Principle}, we firstly re-examine these sets and verify that each $f_i\in I_n$ is a PSD-extremal circuit form. Secondly, we develop a Degree-Jumping Principle as the tool to construct from $I_n$, for any $d\geq 4$, complete sets of separating PSD-extremal circuit forms in $\Delta_{n+1,2d}$. This allows us to answer ($Q^\prime$) by proving the third main result of this paper:

\begin{customthm}{C} \label{thm:allstrict} For $(n+1,2d)_{n\geq 2,d\geq 3}$ and $i=n+1,\ldots,k-n-1$, the strict inclusion $C_i\subsetneq C_{i+1}$ holds. Moreover, if $n\geq 3$, then also the strict inclusion $C_{n}\subsetneq C_{n+1}$ holds.
\end{customthm}

\noindent In Section \ref{properties}, we reconsider the complete sets of separating forms constructed in the proof of Theorem \ref{thm:allstrict} to show that if $n\geq 2$, then $C_{\phi(K_i)}$ is not a spectrahedral shadow for $i=n+2,\ldots,k-n$. Moreover, if $n\geq 3$, then also $C_{\phi(K_{n+1})}$ is not a spectrahedral shadow. This together with Theorem \ref{cor:extendingCi} gives the following fourth main result:

\begin{customthm}{D} \label{thm:spectrahedralshadows} For $(n+1,2d)_{n,d\geq 2}\neq(3,4)$ and $i=n+2,\ldots,k-n$, $C_i$ is not a spectrahedral shadow. Moreover, if $n\geq 3$, then also $C_{n+1}$ is not a spectrahedral shadow.
\end{customthm}

%%%%%%%%%%%%%%%%%%%%%%%%%%%%%%%%%%%%%%%%%%%%%%%%%%%%%%%%%%%%%%%%%%%%%%%%%%%%%%%%%
\section{Non-Negativity on Projective Closures of Affine Varieties} \label{Cone:Description}

The aim of this section is to establish $C_i=C_{\phi(K_i)}$ for $i=0,\ldots,k-n$ by proving Theorem \ref{cor:extendingCi}. To this end, it is crucial to understand when a quadratic form that is locally PSD on $\phi(K_i)(\R)$ can be extended (over the Veronese variety) to a quadratic form that is locally PSD on $V_i(\R)$ (where $V_i$ is known to be the projective closure of the embedded affine variety $\phi(K_i)$). We argue in three steps as follows: \\

\noindent \underline{Step 1}: Recall that any quadratic form $q\in\mathcal{F}_{k+1,2}$ that is locally PSD on an a priori fixed affine set $K\subseteq \R^{k+1}$ is also locally PSD on the Euclidean closure $\overline{K}$ of $K$ in $\R^{k+1}$. We generalize this observation to an embedded affine set $\phi(K)\subseteq\PP^k$ for $K\subseteq\C^k$ and a quadratic form $q\in\mathcal{F}_{k+1,2}$ which is locally PSD on $\phi(K)(\R)$ by giving a sufficient condition for $\phi(K)$, in Lemma \ref{prop:extension}, that allows us to conclude that $q$ is also locally PSD on $\overline{\phi(K)}(\R)$ (w.r.t.\ the Euclidean topology on $\PP^k$). 

\noindent \underline{Step 2}: In the special case when $\phi(K)$ is an embedded affine variety, we explain in the proof of Theorem \ref{cor:conescoincidevariety} that the local PSD property on $\phi(K)(\R)$ even extends to the set of real points of the projective closure of $\phi(K)$.

\noindent \underline{Step 3}: In Proposition \ref{thm:VeroneseCase} and Proposition \ref{thm:KiCase}, for $i=0,\ldots,k-n$, we show that $\phi(K_i)$ satisfies the sufficient condition found in the first step.

\begin{lem} \label{prop:extension} Let $n,d\geq 1$, $K\subseteq \C^k$ and $q\in\mathcal{F}_{k+1,2}$, $q\vert_{_{\phi(K)(\R)}}\geq 0$. If $\overline{\phi(K)(\R)}=\overline{\phi(K)}(\R)$ w.r.t.\ the Euclidean topology on $\PP^k$, then $q\vert_{_{\overline{\phi(K)}(\R)}}\geq 0$.
\end{lem}
    \begin{proof} For $[z]\in\overline{\phi(K)(\R)}\subseteq\PP^k(\R)$, we fix $([z^{(m)}])_{m\in\N}\subseteq\phi(K)(\R)$ and $(\lambda_m)_{m\in\N}\in\C^\times$ such that $\lambda_mz^{(m)}\to z$ as $m\to \infty$. It follows $\mathrm{Re}(\lambda_m)z^{(m)}\to z$ and $\mathrm{Im}(\lambda_m)z^{(m)}\to 0$ as $m\to\infty$, since $(z^{(m)})_{m\in\N}\subseteq\R^{k+1}$ and $z\in\R^{k+1}$. From $q\vert_{_{\phi(K)(\R)}}\geq 0$ and the fact that the quadratic form $q$ is continuous w.r.t.\ the Euclidean topology on $\R^{k+1}$, we conclude $$q(z)=\lim\limits_{m\to\infty}(\mathrm{Re}(\lambda_m))^2q(z^{(m)})\geq 0.$$
    \end{proof}

\begin{cor} \label{cor:conescoincide} Let $n,d\geq 1$ and $K\subseteq \C^k$. If $\overline{\phi(K)(\R)}=\overline{\phi(K)}(\R)$ w.r.t.\ the Euclidean topology on $\PP^k$, then $C_{\phi(K)}=C_{\overline{\phi(K)}}$.
\end{cor}
    \begin{proof} Clear from the definition of $C_{\phi(K)}$, $C_{\overline{\phi(K)}}$ and Lemma \ref{prop:extension}.
    \end{proof}
 
\begin{thm} \label{cor:conescoincidevariety} Let $n,d\geq 1$ and $K\subseteq\C^k$ be an affine variety. If $\overline{\phi(K)(\R)}=\overline{\phi(K)}(\R)$ w.r.t.\ the Euclidean topology on $\PP^k$, then $C_{\phi(K)}=C_{\overline{\phi(K)}}$ w.r.t.\ the Zariski topology on $\PP^k$.
\end{thm}
    \begin{proof} By \cite[(2.33) Theorem]{Mumford}, the Euclidean closure and the Zariski closure of $\phi(K)$ in $\PP^k$ coincide. Hence, by Lemma \ref{prop:extension}, $q\in\mathcal{F}_{k+1,2}$ is locally PSD on the set of real points of the embedded affine variety $\phi(K)\subseteq\PP^k$ if and only if $q$ is locally PSD on the set of real points of the smallest projective variety in $\PP^k$ containing $\phi(K)$. 
    \end{proof}

\begin{prop} \label{thm:VeroneseCase} Let $n,d\geq 1$, then $\overline{\phi(K_{k-n})(\R)}=\overline{\phi(K_{k-n})}(\R)$ w.r.t.\ the Euclidean topology on $\PP^k$.
\end{prop}
    \begin{proof} $(\subseteq)$ Clear. \\
    $(\supseteq)$ Observe that $\overline{\phi(K_{k-n})}(\R)=V(\PP^n)(\R)=V(\PP^n(\R))$ and distinguish two cases for $[x]\in\PP^n(\R)$: \\
    \underline{Case 1}: If $x_0\neq 0$, then $V([x])\in\phi(K_{k-n})(\R)\subseteq\overline{\phi(K_{k-n})(\R)}$. \\
    \underline{Case 2}: If $x_0=0$, then we set $\tilde{x}^{(m)}:=(\frac{1}{m},x_1,\ldots,x_n)$ for $m\in\N$ and observe that $V([\tilde{x}^{(m)}])\in\phi(K_{k-n})(\R)$. Furthermore, $\tilde{x}^{(m)}\to x$ as $m\to\infty$. This implies $(m_0(\tilde{x}^{(m)}),\ldots,m_k(\tilde{x}^{(m)}))\to (m_0(x),\ldots,m_k(x))$ as $m\to\infty$. Hence, $V([x])\in\overline{\phi(K_{k-n})(\R)}$.
    \end{proof}

\begin{prop} \label{thm:KiCase} Let $n,d\geq 1$ and $i=0,\ldots,k-n-1$, then $\overline{\phi(K_i)(\R)}=\overline{\phi(K_i)}(\R)$ w.r.t.\ the Euclidean topology on $\PP^k$.
\end{prop}
    \begin{proof} $(\subseteq)$ Clear. \\
    $(\supseteq)$ Set $H:=\{[z]\in\PP^k \mid \exists\ j\in\{0,\ldots,n+i\} \colon z_j\neq 0\}$ and consider the continuous (w.r.t.\ to the Euclidean topology) map $$\begin{array}{rccc} \pi\colon&\PP^k\cap H &\to& \PP^{n+i} \\ &[z]&\mapsto& [z_0:\ldots:z_{n+i}].
    \end{array}$$
    We assume $\overline{\pi(\phi(K_{k-n}))}(\R)\subseteq\overline{\pi(\phi(K_{k-n})(\R))}$ for now and conclude
    \begin{equation} \label{closureinclusion} \overline{\pi(\phi(K_i))}(\R)\subseteq\overline{\pi(\phi(K_{k-n}))}(\R)\subseteq\overline{\pi(\phi(K_{k-n})(\R))}\subseteq \overline{\pi(\phi(K_i))(\R)}.
    \end{equation} 

    \vspace*{1mm}
    \noindent We now distinguish two cases for $[z]\in \overline{\phi(K_i)}(\R)$ and show $[z]\in\overline{\phi(K_i)(\R)}$: \\
    \underline{Case 1}: If $[z]\not\in H$, then we consider $x^{(m)}:=\left(\frac{1}{m},\ldots,\frac{1}{m}\right)\in\R^{n+1}$ and $y^{(m)}:=(m_0(x^{(m)}),\ldots,m_{n+i}(x^{(m)}),z_{n+i+1},\ldots,z_k)$ for $m\in\N$. Clearly, $[y^{(m)}]\in\phi(K_i)(\R)$ and $y^{(m)}\to z$ as $m\to\infty$. So, $[z]\in\overline{\phi(K_i)(\R)}$. \\
    \underline{Case 2}: If $[z]\in H$, then we fix $([y^{(m)}])_{m\in\N}\subseteq\phi(K_i)$, $(\lambda_m)_{m\in\N}\subseteq\C^\times$ such that $\lambda_my^{(m)}\to z$ as $m\to\infty$. So, $(\pi([y^{(m)}]))_{m\in\N}\subseteq\pi(\phi(K_i))$ and $\lambda_m(y^{(m)}_0,\ldots,y_{n+i}^{(m)})\to (z_0,\ldots,z_{n+i})$ as $m\to\infty$. It consequently follows $\pi([z])\in\overline{\pi(\phi(K_i))}(\R)$. Hence, $\pi([z])\in\overline{\pi(\phi(K_i)(\R))}$ by (\ref{closureinclusion}). This allows us to chose appropriate $([\tilde{y}^{(m)}])_{m\in\N}\subseteq\phi(K_i)(\R)$ and $(\gamma_m)_{m\in\N}\subseteq\C^\times$ such that $\gamma_m(\tilde{y}^{(m)}_0,\ldots,\tilde{y}^{(m)}_{n+i})\to (z_0,\ldots,z_{n+i})$ as $m\to\infty$. In particular, since it holds $((\tilde{y}^{(m)}_0,\ldots,\tilde{y}^{(m)}_{n+i}))_{m\in\N}\subseteq\R^{n+i+1}$ and $(z_0,\ldots,z_{n+i})\in\R^{n+i+1}$, we know that $\mathrm{Re}(\gamma_m)(\tilde{y}^{(m)}_0,\ldots,\tilde{y}^{(m)}_{n+i})\to (z_0,\ldots,z_{n+i})$ as $m\to\infty$. Since $[z]\not\in H$, w.l.o.g.\ $\mathrm{Re}(\gamma_m)\neq 0$ for $m\in\N$. We set $$\overline{y}^{(m)}:=\left(\tilde{y}^{(m)}_0,\ldots,\tilde{y}^{(m)}_{n+i},\frac{1}{\mathrm{Re}(\gamma_m)}z_{n+i+1},\ldots,\frac{1}{\mathrm{Re}(\gamma_m)}z_k\right)$$ and observe that $([\overline{y}^{(m)}])_{m\in\N}\subseteq\phi(K_i)(\R)$. Moreover,   
    $\mathrm{Re}(\gamma_m)\overline{y}^{(m)}\to z$ as $m\to\infty$. This proves $[z]\in\overline{\phi(K_i)(\R)}$.
    
    \vspace*{1mm}
    \noindent It remains to verify $\overline{\pi(\phi(K_{k-n}))}(\R)\subseteq\overline{\pi(\phi(K_{k-n})(\R))}$: \\
    \noindent Let $\overline{(\cdot)}^H$ denote the Euclidean (subspace) closure of the indicated subset in $\PP^k\cap H$. The continuity of $\pi$ (w.r.t.\ the Euclidean topology) and the fact that $\overline{\phi(K_{k-n})}^H=V(\PP^n)\cap H$ together imply $$\overline{\pi(\phi(K_{k-n}))}(\R)\subseteq\pi(\overline{\phi(K_{k-n})}^H)(\R)=\pi(V(\PP^n)\cap H)(\R).$$ If $\pi(V(\PP^n)\cap H)(\R)\subseteq \pi((V(\PP^n)(\R)\cap H)$, then Proposition \ref{thm:VeroneseCase} together with the continuity of $\pi$ yield
    $$\overline{\pi(\phi(K_{k-n}))}(\R)\subseteq \pi(\overline{\phi(K_{k-n})}(\R)\cap H)\subseteq\pi(\overline{\phi(K_{k-n})(\R)}^H)\subseteq \overline{\pi(\phi(K_{k-n})(\R))}.$$

    \vspace*{1mm}
    \noindent It thus remains to prove $\pi(V(\PP^n)\cap H)(\R)\subseteq \pi((V(\PP^n)(\R)\cap H)$: \\
    \noindent For $[z]\in V(\PP^n)\cap H$ with $(z_0,\ldots,z_{n+i})\in\R^{n+i+1}$, we fix some $x\in\C^{n+1}$ such that $(z_0,\ldots,z_{n+i})=(m_0(x),\ldots,m_{n+i}(x))$ and observe that there exists $0\leq j \leq n$ with $X_j^d\geq_{lex} m_{n+i}(X)$ and $x_j\neq 0$, since $[z]\in H$. We set $$j:=\min\{j\in\{0,\ldots,n\} \mid X_j^d\geq_{lex} m_{n+i}(X) \wedge x_j\neq 0\}$$ and let $0\leq s \leq n+i$ be such that $m_s(X)=X_j^d$. Hence, for $l=0,\ldots,j-1$, $x_l=0$ and $0\neq m_s(x)=x_j^d=z_s$. W.l.o.g.\ assume $z_s>0$ and fix $y\in\R^\times$ such that $z_s=y^d$. We distinguish two cases for $s$ and construct $[\tilde{x}]\in\PP^n(\R)$ such that
    $\pi([z])=\pi(V([\tilde{x}]))$: 
    
    \noindent\underline{Case 1:} If $s=n+i$, then we set $\tilde{x}$ to be the vector in $\R^{n+1}$ with $\tilde{x}_s:=y$ and zeros elsewhere. Consequently, $(z_0,\ldots,z_{n+i})=(m_0(\tilde{x}),\ldots,m_{n+i}(\tilde{x}))$. \\
    \noindent\underline{Case 2:} If $s<n+i$, then we set $r:=\min\{n-j,n+i-s\}$. For $l=1,\ldots,r$, we have $z_{s+l}=x_j^{d-1}x_{j+l}=x_j^d\frac{x_{j+l}}{x_j}\in\R$. Since $x_j^d=z_s\in\R$, it follows that $\lambda_l:=\frac{x_{j+l}}{x_j}\in\R$. We conclude $(z_0,\ldots,z_{n+i})=(m_0(\tilde{x}),\ldots,m_{n+i}(\tilde{x}))$ for $\tilde{x}:=(\tilde{x}_0,\ldots,\tilde{x}_n)\in\R^{n+1}$ defined by 
    $$\tilde{x}_t:=\begin{cases} y, & \mathrm{for\ } t=j \\
                                \lambda_{t-j} y, & \mathrm{for\ } t\in\{j+1,\ldots,j+r\} \\
                                0, & \mathrm{else} \end{cases}$$
    \end{proof}

\noindent Theorem \ref{cor:extendingCi} now follows from Theorem \ref{cor:conescoincidevariety} and Propositions \ref{thm:VeroneseCase}, \ref{thm:KiCase}.
%%%%%%%%%%%%%%%%%%%%%%%%%%%%%%%%%%%%%%%%%%%%%%%%%%%%%%%%%%%%%%%%%%%%%%%%%%%%%%%%%
\section{Computing $i(f)$ for PSD-extremal Circuit Forms} \label{if}

Recall that for $f\in\Delta_{n+1,2d}$, there exists a unique $i(f)\in\{0,\ldots,k-n-1\}$ such that $f\in C_{i(f)+1}\backslash C_{i(f)}$. The aim of this section is to prove the second main result (Theorem \ref{cor:greatestcone}) of this paper by computing $i(f)$ for a PSD-extremal $f\in\Delta_{n+1,2d}^{\mathfrak{C}}$. To this end, for $i=0,\ldots, k-n$ and $f\in\mathcal{P}_{n+1,2d}$, we firstly state and prove a sufficient condition in Theorem \ref{thm:suffCi} that if $f\in \mathfrak{S}_{n+i}$ (recall that $\mathfrak{S}_{n+i}:=\mathrm{span}_{\R}\{m_sm_t \mid 0\leq s,t\leq n+i\}$), then $f\in C_{i}$. In Proposition \ref{prop:circuitsatisfysuff}, we then verify that if $f$ is a circuit form (with outer exponents $2\alpha_{j_0},\ldots,2\alpha_{j_r}$), then $f\in\mathfrak{S}_{j(f)}$ (recall that $j(f):=\max\{j_0,\ldots,j_r\}$). This allows us to compute an upper bound for $i(f)$ in Theorem \ref{thm:Containment} and Corollary \ref{cor:upperbound}. Secondly, we show, in Theorem \ref{thm:exclusion}, that this upper bound is sharp if $f$ is PSD-extremal.

\begin{thm} \label{thm:suffCi} Let $n,d\geq 1$ and $i=0,\ldots,k-n$, then $$C_{i}\cap \mathfrak{S}_{n+i}=\mathcal{P}_{n+1,2d}\cap \mathfrak{S}_{n+i}.$$
\end{thm}
    \begin{proof} ($\subseteq$) Clear. \\
    ($\supseteq$) For $f\in\mathcal{P}_{n+1,2d}\cap\mathfrak{S}_{n+i}$, we fix a Gram matrix $A:=(a_{s,t})_{0\leq s,t\leq k}$ associated to $f$ such that $a_{s,t}= 0$ whenever $n+i<\max\{s,t\}$. For $[z]\in\phi(K_{i})(\R)$, we fix $x\in\R^{n+1}$ such that $(z_0,\ldots,z_{n+i})=(m_0(x),\ldots,m_{n+i}(x))$. We compute $q_{A}(z)=f(x)\geq 0$ and conclude $f\in C_{\phi(K_i)}=C_i$ by Theorem \ref{cor:extendingCi}.
    \end{proof}

\begin{prop} \label{prop:circuitsatisfysuff} Let $n,d\geq 1$ and $f\in\mathfrak{C}_{n+1,2d}$, then $f\in\mathfrak{S}_{j(f)}$.
\end{prop}
    \begin{proof} We verify the assertion by an induction on $d\geq 1$. The inductive base case $d=1$ is clear. Thus, we assume that the assertion was already verified up to some $d\geq 1$ and investigate the situation for $d+1$. If $j(f)\leq k(n,d)$, then $f(X)=X_0^2g(X)$ for some $g\in\mathfrak{C}_{n+1,2d}$. Since the inductive assumption can be applied to $g$ and $f(X)=X_0^2g(X)$, $f\in\mathfrak{S}_{j(f)}$ follows. 
    
    However, for $j(f)\geq k(n,d)+1$, we set $$l(f):=\max\{l\in\{1,\ldots,n\}\mid X_l^d\geq_{lex} m_{j(f)}(X)\}$$ and argue by an induction on $n\geq 1$ that $f$ belongs to $\mathfrak{S}_{j(f)}$. For the inductive base case, $n=1$ and $j(f)=k(n,d+1)$. The assertion thus immediately follows. We therefore assume that the assertion was already verified up to some $n\geq 1$ (i.e., for circuit forms in up to $n+1$ variables of degree $2(d+1)$) and investigate the situation for $n+1$. 

    For $f\in\mathfrak{C}_{n+2,2(d+1)}$, we write $$f(X)=\sum\limits_{l=0}^r f_{2\alpha_{j_l}}X^{2\alpha_{j_l}}+f_bX^{b}$$ where $b=2\sum\limits_{l=0}^r \lambda_l \alpha_{j_l}$ for some $\lambda_0,\ldots,\lambda_r>0$ such that $\sum\limits_{l=0}^r \lambda_l = 1$. If $j_0,\ldots,j_r\geq k(n+1,d)+1$, then $f$ can be interpreted as a circuit form in $n+1$ variables of degree $2(d+1)$ and the assertion follows by the inductive assumption. Therefore, we now assume that there exists some $0\leq \sigma \leq r$ such that $0\leq j_\sigma\leq k(n+1,d)$. Consequently, $X_0$ divides $X^b$. If $X_0^2$ divides $X^b$, then we are done. Hence, we now only consider the case when $X^b=X_0m(X_1,\ldots,X_{n+1})$ for some monomial $m$ of degree $2d+1$. 

    If $l(f)\geq 2$ and $X_1$ divides $X^b$, then we are done. However, if $l(f)\geq 2$ and $X_1$ does not divide $X^b$, then $f$ can be interpreted as a circuit form in $n+1$ variables of degree $2(d+1)$ and we are done by the inductive assumption.
    Therefore, we now assume $l(f)=1$ and set $I:=\{l\in\{0,\ldots,r\} \mid \alpha_{j_l,0}\geq 1\}$ (which is non-empty since $X_0$ divides $X^b$) and $I^\prime:=\{0,\ldots,r\}\backslash I$. It follows
    \begin{equation} \label{b0} 1=b_0= 2\sum\limits_{l\in I}\lambda_l\alpha_{j_l,0}\geq 2\sum\limits_{l\in I}\lambda_l.
    \end{equation} Consequently, $\sum\limits_{l\in I^\prime}\lambda_l\geq \frac{1}{2}\geq \sum\limits_{l\in I}\lambda_l$. Let us distinguish three cases:

    \noindent \underline{Case 1}: If $\alpha_{j_l,1}\geq 1$ for some $l\in I$, then we are done since $$b_1 > 2\sum\limits_{l\in I^\prime} \lambda_l\alpha_{j_l,1} \geq 2\left(\sum\limits_{l\in I^\prime}\lambda_l\right)\alpha_{j(f),1}\geq \alpha_{j(f),1}.$$
    \noindent\underline{Case 2}: If $\alpha_{j_l,1}>\alpha_{j(f),1}$ for some $l\in I^\prime$, then we are done since $$b_1\geq 2\sum\limits_{l\in I^\prime}\lambda_l\alpha_{j_l,1} > 2\left(\sum\limits_{l\in I^\prime}\lambda_l\right)\alpha_{j(f),1}\geq \alpha_{j(f),1}.$$
    \noindent \underline{Case 3}: If we are in neither of the above two cases, then $$b_1=2\sum\limits_{l\in I^\prime}\lambda_l\alpha_{j_l,1}=2\left(\sum\limits_{l\in I^\prime}\lambda_l\right)\alpha_{j(f),1}\geq\alpha_{j(f),1}.$$ If the inequality is strict, then we are done. Else, $\sum\limits_{l\in I^\prime}\lambda_l=\frac{1}{2}=\sum\limits_{l\in I}\lambda_l$ and $\alpha_{j_l,0}=1$ for $l\in I$ by (\ref{b0}). Summing up, $(\alpha_{j_l,0},\alpha_{j_l,1})=(1,0)$ for $l\in I$ and $(\alpha_{j_l,0},\alpha_{j_l,1})=(0,\alpha_{j(f),1})$ for $l\in I^\prime$. We iterate the above argument for $b_2,\ldots,b_{n+1}$ and necessarily terminate in one of the solved first two cases. Else, $\alpha_{j_l}=(1,0,\ldots,0)$ for $l\in I$ which is impossible since $d+1\geq 2$.
    \end{proof}
    
\begin{thm} \label{thm:Containment} Let $n,d\geq 1$ and $f\in\mathcal{P}_{n+1,2d}^{\mathfrak{C}}$. The following are true:
    \begin{enumerate} 
    \item If $0\leq j(f)\leq n-1$, then $f\in\Sigma_{n+1,2d}^\mathfrak{C}$.
    \item If $n\leq j(f)\leq k(n,d)$, then $f\in C_{j(f)-n}$.
    \end{enumerate}
\end{thm}
    \begin{proof} Proposition \ref{prop:circuitsatisfysuff} yields $f\in\mathfrak{S}_{j(f)}$.
    \begin{enumerate}
    \item If $0\leq j(f)\leq n-1$, then $\mathfrak{S}_{j(f)}\subseteq\mathfrak{S}_{n+1}$. Hence, $f(X)=X_0^{2d-2}g(X)$ for some $g\in\mathcal{P}_{n+1,2}^\mathfrak{C}=\Sigma_{n+1,2}^\mathfrak{C}$. Consequently, $f\in\Sigma_{n+1,2d}^\mathfrak{C}$.
    \item If $n\leq j(f)\leq k(n,d)$, then $f\in\mathcal{P}_{n+1,2d}\cap\mathfrak{S}_{j(f)}$ yields $f\in C_{j(f)-n}\cap\mathfrak{S}_{j(f)}$ by Theorem \ref{thm:suffCi}. 
    \end{enumerate}
    \end{proof}

\begin{cor} \label{cor:upperbound} Let $n,d\geq 2$, $(n+1,2d)\neq (3,4)$ and $f\in\Delta_{n+1,2d}^{\mathfrak{C}}$, then $i(f)\leq j(f)-n-1$ and the following are true:
\begin{enumerate}
    \item If $n=2$, then $j(f)\geq 6$.
    \item If $n\geq 3$, then $j(f)\geq 2n+1$.
\end{enumerate}
\end{cor}
\begin{proof} Clear from \cite[Theorem A]{GHK} and Theorem \ref{thm:Containment}.
\end{proof}

\begin{thm} \label{thm:exclusion} Let $n,d\geq 1$ and $f\in\Delta_{n+1,2d}^\mathfrak{C}$. If $f$ is PSD-extremal,
     then $i(f)\geq j(f)-n-1$.
\end{thm}
    \begin{proof} W.l.o.g.\ let $f_{j(f)}=1$ and assume that $i(f)< j(f)-n-1$ for a contradiction. So, $f\in C_{j(f)-n-1}$ which allows us to fix some Gram matrix $A:=(a_{i,j})_{0\leq i,j\leq k}$ associated to $f$ such that $q_A$ is locally PSD on $V_{j(f)-n-1}(\R)$. We now show that $(a_{i,j})_{0\leq i \leq j}\equiv 0$ for any $j\geq j(f)+1$: \\

    Indeed, if $j(f)=k$, then there is nothing to prove. Thus, w.l.o.g.\ we assume $j(f)<k$. By comparision of coefficients between $q_A(m_0(X),\ldots,m_k(X))$ and $f(X)$, we have that $a_{k,k}=0$. Moreover, for any $x\in\R^{n+1}$, $y\in\R$, $y\neq 0$, it follows from $[m_0(x):\ldots:m_{k-1}(x):y]\in V_{j(f)-n-1}(\R)$ that $0\leq q_A(m_0(x),\ldots,m_{k-1}(x),z)$. Hence, for some appropriate $c\in\R$, $$q_x(Z):=q_A(m_0(x),\ldots,m_{k-1}(x),Z)=\left(2\sum\limits_{i=0}^{k-1} a_{i,k}m_i(x)\right)Z+c$$ is a PSD univariate polynomial (for any fixed $x\in\R^{n+1}$) by continuity. Consequently, $a_{0,k}=\ldots=a_{k-1,k}=0$. Iteratively, the assertion follows. \\

    \noindent Since $f_{j(f)}=1$, $a_{j(f),j(f)}=1$ thus follows and we compute over $\R$ that
    \begin{eqnarray*}
        0&\leq & q_A(m_0(X),\ldots,m_{j(f)-1}(X),Z,m_{j(f)+1}(X),\ldots,m_{k}(X)) \\
        &=&q_A(m_0(X),\ldots,m_k(X))-2\left(\sum\limits_{i=0}^{j(f)-1}a_{i,j(f)}m_i(X)\right)(m_{j(f)}(X)-Z)\\
        &&-m_{j(f)}(X)^2+Z^2 \\
        &=&f(X)+\left(Z+\sum\limits_{i=0}^{j(f)-1}a_{i,j(f)}m_i(X)\right)^2-\left(\sum\limits_{i=0}^{j(f)-1}a_{i,j(f)}m_i(X)+m_{j(f)}(X)\right)^2.
    \end{eqnarray*}
    For $Z:=-\sum\limits_{i=0}^{j(f)-1}a_{i,j(f)}m_i(X)$, we obtain
    $$0\leq f(X)-\left(\sum\limits_{i=0}^{j(f)-1}a_{i,j(f)}m_i(X)+m_{j(f)}(X)\right)^2.$$ Hence, $f=f_1+f_2$ for the PSD forms 
    \begin{eqnarray*}
        f_1(X)&:=&\frac{1}{2}\left(f(X)-\left(\sum\limits_{i=0}^{j(f)-1}a_{i,j(f)}m_i(X)+m_{j(f)}(X)\right)^2\right), \\
        f_2(X)&:=&\frac{1}{2}\left(f(X)+\left(\sum\limits_{i=0}^{j(f)-1}a_{i,j(f)}m_i(X)+m_{j(f)}(X)\right)^2\right).
    \end{eqnarray*} Since $f$ is assumed to be PSD-extremal, we can thus fix $\lambda_1\geq 0$ such that $f_1=\lambda_1 f$. The monomial square $m_{j(f)}(X)^2$ has zero-coefficient on the left hand side, i.e., in $f_1$, and coefficient one in $f$. Thus, $\lambda_1=0$. This implies that $f_1$ is the zero form and, consequently, $$f(X)=\left(\sum\limits_{i=0}^{j(f)-1}a_{i,j(f)}m_i(X)+m_{j(f)}(X)\right)^2.$$ A contradiction to the assumption that $f$ is not SOS.
    \end{proof}

\noindent Theorem \ref{cor:greatestcone} thus follows from Corollary \ref{cor:upperbound} and Theorem \ref{thm:exclusion}.
However, this theorem only provides a sufficient condition to compute $i(f)$. Yet, it does not give a necessary condition as there exists a $f\in \Delta_{n+1,2d}^{\mathfrak{C}}$ such that $i(f)=j(f)-n-1$ which is not PSD-extremal as shown below:

\begin{exmp} \label{exmp:notPSDextremalform}
     Consider the PSD quaternary sextic circuit form $$\mathfrak{D}(X_0,X_1,X_2,X_3):=2X_0^4X_3^2+2X_0^2X_3^4+X_1^4X_2^2+X_1^2X_2^4-6X_0^2X_1X_2X_3^2$$ which is not PSD-extremal \cite[(9.9) Example]{Rez2}. We compute $j(\mathfrak{D})=13$. By Theorem \ref{thm:Containment}, $\mathfrak{D}\in C_{10}$. However, Theorem \ref{thm:exclusion} cannot be applied for concluding $\mathfrak{D}\not\in C_9$ since $\mathfrak{D}$ is not PSD-extremal. Yet, for any $A\in\mathcal{G}^{-1}(\mathfrak{D})$, various point evaluations of $q_A$ on $V_9(\R)$ imply that $\mathfrak{D}\not\in C_9$. It follows $i(\mathfrak{D})=9=j(\mathfrak{D})-n-1$.
\end{exmp} 
%%%%%%%%%%%%%%%%%%%%%%%%%%%%%%%%%%%%%%%%%%%%%%%%%%%%%%%%%%%%%%%%%%%%%%%%%%%%%%%%%
\section{Separation of the Cones $C_{n+1},\ldots,C_{k-n}$ and also $C_n$ if $n\geq 3$} \label{Degree-Jumping Principle}
The aim of this section is to prove the third main result, Theorem \ref{thm:allstrict}, of this paper. To this end, we start off by recalling two main results from \cite{GHK}:

\begin{thm} \cite[Theorem C]{GHK} \label{exmp:naryquartics}
    For $(n+1,4)_{n\geq 4}$ and $i=n,\ldots,k-n-1$, the strict inclusion $C_i\subsetneq C_{i+1}$ holds.
\end{thm}

\begin{thm} \cite[Theorem D]{GHK} \label{exmp:narysextics}
    For $(n+1,6)_{n\geq 3}$ and $i=n,\ldots,k-n-1$, the strict inclusion $C_{i} \subsetneq C_{i+1}$ holds.
\end{thm}

In the proofs of the above two theorems in \cite{GHK}, in each case, we construct a set $I_{n,d}\subseteq\Delta_{n+1,2d}$ that we verify to be a complete set of separating forms by establishing the separating property of each $f\in I_{n,d}$ via several point evaluations. We now reconsider the sets $I_{n,d}$ and show that each $f\in I_{n,d}$ is a PSD-extremal circuit form by using Theorem \ref{cor:greatestcone}. This gives an alternative argument for $I_{n,d}$ to be a complete set of separating forms.

\begin{proof}[Proof of Theorem \ref{exmp:naryquartics}]
    We distinguish three cases:
	\begin{enumerate}
		\item $m_{n+i}(X)=X_jX_n$ for some $1\leq j \leq n-2$
		\item $m_{n+i}(X)=X_jX_l$ for some $2\leq j\leq l\leq n-1$
		\item $m_{n+i}(X)=m_{k-1}(X)$ 
	\end{enumerate}
	In each case, we propose the following separating forms in $C_{i+1}\backslash C_i$:
	\begin{enumerate}
		\item $\mathfrak{c}^\sigma(X):= X_0^2X_n^2+X_0^2X_{j}^2+X_{j}^2X_n^2+X_{j+1}^4-4X_0X_{j}X_{j+1}X_n$ 
		\item $\mathfrak{c}^\tau(X):=X_1^2X_{l+1}^2+X_j^2X_{l+1}^2+X_1^2X_j^2+X_0^4-4X_0X_1X_jX_{l+1}$
		\item $\mathfrak{c}(X):=X_0^2X_1^2+X_0^2X_{n-1}^2+X_1^2X_{n-1}^2+X_n^4-4X_0X_1X_{n-1}X_n$
	\end{enumerate}
	The $(n+1)$-ary quartics $\mathfrak{c}$, $\mathfrak{c}^\sigma$ and $\mathfrak{c}^\tau$ are interpretations of the Choi-Lam quaternary quartic \cite{CL2,CL} $$\mathfrak{C}(Y_0,Y_1,Y_2,Y_3):=Y_0^2Y_1^2+Y_0^2Y_2^2+Y_1^2Y_2^2+Y_3^4-4Y_0Y_1Y_2Y_3$$ which is a PSD-extremal (\cite[Theorem 3.5]{CL2}) circuit form that is not SOS. Thus, so are $\mathfrak{c}$, $\mathfrak{c}^\sigma$, and $\mathfrak{c}^\tau$. Moreover, in each case, $j(f)=n+i+1$ for $f\in\{\mathfrak{c}^\sigma,\mathfrak{c}^\tau,\mathfrak{c}\}$. Hence, by Theorem \ref{cor:greatestcone}, $i=i(f)$ in each case for $f\in\{\mathfrak{c}^\sigma,\mathfrak{c}^\tau,\mathfrak{c}\}$. Thus, these distinguished circuit quartics separate the corresponding cones $C_i$ and $C_{i+1}$ in each case. In particular, we constructed a complete set $I_{n,2}\subseteq\Delta_{n+1,4}$ of separating forms such that $\vert I_{n,2} \vert = \mu(n,2)+1$ and any $f\in I_{n,2}$ is a PSD-extremal circuit form.
\end{proof}

\begin{proof}[Proof of Theorem \ref{exmp:narysextics}]
    We denote the complete set of separating forms in $\Delta_{n+1,4}$ from the above proof of Theorem \ref{exmp:naryquartics} by $I_{n,2}$. In particular, we know that $\vert I_{n,2} \vert = \mu(n,2)+1$ and any $g\in I_{n,2}$ is a PSD-extremal circuit form. Hence, for any $g\in I_{n,2}$, also $X_0^2g(X)\in\Delta_{n+1,6}$ is a PSD-extremal circuit form. For the verification of $C_i\subsetneq C_{i+1}$ for $i=k(n,2)-n,\ldots,k(n,3)-n-1$, we distinguish five cases:
	\begin{enumerate}
		\item $m_{n+i}(X)=m_{k(n,2)}(X)$
		\item $m_{n+i}(X)=X_j^2X_l$ for some $1\leq j\leq l\leq n-1$
		\item $m_{n+i}(X)=X_jX_lX_n$ for some $1\leq j\leq l\leq n-1$
		\item $m_{n+i}(X)=X_jX_n^2$ for some $1\leq j\leq n-1$
		\item $m_{n+i}(X)=X_jX_lX_r$ for some $1\leq j< l \leq r \leq n-1$
	\end{enumerate}
	In each case, we propose the following separating forms in $C_{i+1}\backslash C_i$:
	\begin{enumerate}
		\item $\mathfrak{m}^\sigma(X):= X_0^4X_n^2+X_0^2X_n^4+X_1^6-3X_0^2X_1^2X_n^2$
		\item $\mathfrak{l}(X):=X_0^4X_j^2+X_0^2X_{l+1}^4+X_j^4X_{l+1}^2-3X_0^2X_j^2X_{l+1}^2$
		\item $\mathfrak{l}^\tau(X):=X_0^4X_{l+1}^2+X_0^2X_j^4+X_{l+1}^4X_j^2-3X_0^2X_j^2X_{l+1}^2$
		\item $\mathfrak{m}(X):=X_0^4X_j^2+X_0^2X_j^4+X_{j+1}^6-3X_0^2X_j^2X_{j+1}^2$
		\item $\mathfrak{c}^\rho(X):=X_j^4X_{r+1}^2+X_j^2X_l^2X_{r+1}^2+X_j^4X_l^2+X_0^4X_j^2-4X_0X_j^3X_lX_{r+1}$
	\end{enumerate}
	The $(n+1)$-ary sextics $\mathfrak{m}$ and $\mathfrak{m}^\sigma$ are interpretations of the Motzkin form $\mathfrak{M}$ \cite{Motzkin}, while $\mathfrak{l}$ and $\mathfrak{l}^\tau$ are interpretations of the Choi-Lam ternary sextic $\mathfrak{L}$ \cite{CL2, CL}.
    Likewise, $\mathfrak{c}^\rho$ comes from a multiplication of the Choi-Lam quaternary quartic $\mathfrak{C}$ (see proof of Theorem \ref{exmp:naryquartics}) with the monomial square $X_j^2$. $\mathfrak{M}$, $\mathfrak{L}$ and $\mathfrak{C}$ are PSD-extremal (\cite[Corollary 3.3, Theorem 3.4 and Theorem 3.5]{CL2}) circuit forms that are not SOS. Thus, so are $\mathfrak{m}$, $\mathfrak{m}^\sigma$, $\mathfrak{l}$, $\mathfrak{l}^\tau$ and $\mathfrak{c}^\rho$. Moreover, in each case, $j(f)=n+i+1$ for $f\in\{\mathfrak{m}^\sigma,\mathfrak{l},\mathfrak{l}^\tau,\mathfrak{m},\mathfrak{c}^\rho\}$. Hence, $i=i(f)$ in each case for $f\in\{\mathfrak{m}^\sigma,\mathfrak{l},\mathfrak{l}^\tau,\mathfrak{m},\mathfrak{c}^\rho\}$ by Theorem \ref{cor:greatestcone}. Thus, these distinguished circuit sextics separate the corresponding cones $C_i$ and $C_{i+1}$ in each case. In other words, recalling \cite[Theorem B]{GHK}, the union $I_{n,3}$ of the collection of the above distinguished circuit sextics and $\{X_0^2g(X) \mid g\in I_{n,2}\}$ is a complete set of separating forms in $\Delta_{n+1,6}$ such that $\vert I_{n,3} \vert = \mu(n,3)+1$ and any $f\in I_{n,3}$ is a PSD-extremal circuit form. 
\end{proof}

\begin{rem} \label{rem:sepset} In the proof of \cite[Theorem 4.2]{GHK}, for $d=3$, $n=2$, we find a set $J_{2,3}\subseteq \Delta_{3,6}^{\mathfrak{C}}$ of PSD-extremal forms such that $$\{i(f) \mid f\in J_{2,3}\}=\{k(2,2)-2,\ldots,k(2,3)-2-1\}.$$ Likewise, in the proof of \cite[Theorem C]{GHK} given above, for $d=3$, $n\geq 3$, we especially find a set $J_{n,3}\subseteq \Delta_{n+1,6}^{\mathfrak{C}}$ of PSD-extremal forms such that $$\{i(f) \mid f\in J_{n,3}\}=\{k(n,2)-n,\ldots,k(n,3)-n-1\}.$$
\end{rem} 
For $n\geq 2$ and any $d\geq 4$, we now establish a Degree-Jumping Principle in Proposition \ref{thm:degreejump} that allows us to generate 
a set $J_{n,d}\subseteq\Delta_{n+1,2d}^{\mathfrak{C}}$ of PSD-extremal forms from $J_{n,3}$ such that $$\{i(f) \mid f\in J_{n,d}\} = \{k(n,d-1)-n,\ldots,k(n,d)-n-1\}.$$ Lastly, we extend $J_{n,d}$ to a complete sets of separating forms $I_{n,d}\subseteq\Delta_{n+1,2d}$ in the proof of Theorem \ref{thm:allstrict} below.

\begin{prop}[Degree-Jumping Principle] \label{thm:degreejump} $\;$ \\
Let $n\geq 2$ and assume that there exists a set $J_{n,3}\subseteq\Delta_{n+1,6}^{\mathfrak{C}}$ of PSD-extremal forms such that 
$$\{i(f) \mid f\in J_{n,3}\}=\{k(n,2)-n,\ldots,k(n,3)-n-1\}.$$ Then, for $d\geq 4$, there exists a set $J_{n,d}\subseteq\Delta_{n+1,2d}^{\mathfrak{C}}$ of PSD-extremal forms such that $$\{i(f) \mid f\in J_{n,d}\}=\{k(n,d-1)-n,\ldots,k(n,d)-n-1\}.$$
\end{prop}
    \begin{proof}
         We verify the assertion by an induction on $d\geq 3$. The inductive base case $d=3$ is solved by assumption. Thus, we assume that the assertion was already verified up to some $d\geq 3$ and investigate the situation for $d+1$. For the purpose of this proof, we denote the lexicographically ordered monomial basis of $\mathcal{F}_{n+1,\delta}$ by 
         $\{m_0^{(\delta)},\ldots,m_{k(n,\delta)}^{(\delta)}\}$ and the corresponding exponents by $\alpha_0^{(\delta)},\ldots,\alpha_{k(n,\delta)}^{(\delta)}$ for $\delta\in\{d,d+1\}$. 
         
         For $i=k(n,d)-n,\ldots,k(n,d+1)-n-1$, we set $$l:=\max\{l\in\{1,\ldots,n\}\mid X_l^{d+1}\geq_{lex} m_{n+i+1}^{(d+1)}(X)\}$$ and fix $j\in\{k(n,d)-k(n-l,d)-n-1,\ldots,k(n,d)-n-1\}$ such that $X_{l}m_{n+j+1}^{(d)}(X)=m_{n+i+1}^{(d+1)}(X)$. The inductive assumption yields the existence of a PSD-extremal form $g\in\Delta_{n+1,2d}^\mathfrak{C}$ such that $i(g)=j$. For the
         PSD-extremal form $f(X):=X_l^2g(X)\in\Delta_{n+1,2(d+1)}^\mathfrak{C}$, we compute by Theorem \ref{cor:greatestcone} that $i(f)=j(f)-n-1=j(X_l^2g(X))-n-1=i$.
    \end{proof}

\begin{cor} \label{cor:degreejump} For $n\geq 2$, $d\geq 3$, there exists a set $J_{n,d}\subseteq\Delta_{n+1,2d}^{\mathfrak{C}}$ of PSD-extremal forms such that $$\{i(f) \mid f\in J_{n,d}\}=\{k(n,d-1)-n,\ldots,k(n,d)-n-1\}.$$
\end{cor}
    \begin{proof} The assertion follows from Remark \ref{rem:sepset} and Proposition \ref{thm:degreejump}.
    \end{proof}

    \begin{proof}[\textbf{Proof of Theorem \ref{thm:allstrict}}] We fix $n\geq 2$ and argue inductively over $d\geq 3$. If $d=3$, then the assertion is true by \cite[Theorem 4.2 and Theorem D]{GHK}. We therefore assume that the statement was already verified up to some $d\geq 3$. We now investigate the situation for $d+1$. 
    
    By the inductive assumption, each inclusion in $C_{n+1}\subseteq \ldots \subseteq C_{k(n,d)-n}$ as subcones of $\mathcal{P}_{n+1,2d}$ is strict. Thus, each inclusion in $C_{n+1}\subseteq \ldots \subseteq C_{k(n,d)-n}$ as subcones of $\mathcal{P}_{n+1,2d+2}$ is strict by \cite[Theorem B]{GHK}. The same is true for the inclusion $C_n\subseteq C_{n+1}$ when $n\geq 3$. Moreover, Corollary \ref{cor:degreejump} implies that each inclusion in $C_{k(n,d)-n}\subseteq \ldots \subseteq C_{k(n,d+1)-n}$ is strict.
    \end{proof}

The proof of Theorem \ref{thm:allstrict} provides a method to explicitly construct,  complete sets $I_{n,d}\subseteq\Delta_{n+1,2d}^{\mathfrak{C}}$ of separating forms such that any $f\in I_{n,d}$ is PSD-extremal and $\vert I_{n,d} \vert = \mu(n,d)+1$. In fact, even the following holds:

\begin{cor} \label{cor:specialf} For $n\geq 2$, $d\geq 3$ and $i=n+1,\ldots,k-n-1$, there exists (a PSD-extremal circuit form) $f_i\in (C_{i+1}\backslash C_i)\cap\mathfrak{S}_{n+i+1}$. Moreover, if $n\geq 3$, then there also exists (a PSD-extremal circuit form) $f_n\in (C_{n+1}\backslash C_{n})\cap\mathfrak{S}_{2n+1}$.
\end{cor}

We thus completely answered ($Q$) by providing a refinement of Hilbert's 1888 Theorem with the number of strictly separating cones in (\ref{ConeFiltration}) being
$$\mu(n,d):=\begin{cases} (k(n,d)-n)-(n+1)-1, & \mbox{for } (3,2d)_{d\geq 3} \\
	(k(n,d)-n)-(n+1), & \mbox{for } (n+1,2d)_{n\geq 3, d\geq 2}. \end{cases}$$
 The filtration in (\ref{ConeFiltration}), thus, reduces to a specific cone subfiltration \begin{equation} \label{part1:eq2} \Sigma_{n+1,2d}=C_0^\prime\subsetneq C_1^\prime \subsetneq \ldots \subsetneq C_{\mu(n,d)}^\prime \subsetneq C_{\mu(n,d)+1}^\prime=\mathcal{P}_{n+1,2d},
\end{equation} in which each inclusion is strict.
%%%%%%%%%%%%%%%%%%%%%%%%%%%%%%%%%%%%%%%%%%%%%%%%%%%%%%%%%%%%%%%%%%%%%%%%%%%%%%%%%%%%%%%%%%%%%%%%%%%%%%%%%%%%%%%%%%%%%%%%%%%%%%%%
\section{(Non-)Spectrahedral Shadows} \label{properties}
In the Hilbert cases, $\Sigma_{n+1,2d}=C_0=\ldots=C_{k-n}=\mathcal{P}_{n+1,2d}$ and $\Sigma_{n+1,2d}$ is a spectrahedral shadow. Thus, all $C_i$'s are spectrahedral shadows. In any non-Hilbert case, a similar argument is valid only for the first $n+1$ (respectively $n+2$ if $n=2$) cones in (\ref{ConeFiltration}) as we show in Lemma \ref{lem:ss}. However, for the remaining $C_i$'s, the situation is the opposite, i.e., each strictly separating $C_i$ in (\ref{ConeFiltration}) fails to be a spectrahedral shadow. This is the last main result (Theorem \ref{thm:spectrahedralshadows}) of this paper which we prove by an induction over $d$ using Theorem \ref{thm:notss} and Theorem \ref{thm:notssjump}.

\begin{lem} \label{lem:ss} For $n,d\geq 2$, $(n+1,2d)\neq (3,4)$ and $i=0,\ldots,n$, $C_i$ is a spectrahedral shadow. Moreover, if $n=2$, then also $C_{n+1}$ is a spectrahedral shadow.
\end{lem}
    \begin{proof} By \cite[Theorem A]{GHK}, $C_i=\Sigma_{n+1,2d}$ which is a spectrahedral shadow.
    \end{proof}

\begin{thm} \label{thm:notss} The following are true:
    \begin{enumerate}
        \item For $n\geq 3$, $d=2$ and $i=n+1,\ldots,k-n$, $C_i$ is not a spectrahedral shadow.
        \item For $n\geq 2$, $d\geq 3$ and $i=k(n,d-1)-n+1,\ldots,k(n,d)-n$, $C_i$ is not a spectrahedral shadow.
    \end{enumerate}
\end{thm}

\begin{proof} Set $\{e_0,\ldots,e_N\}:=\{m_sm_t \mid 0 \leq s,t \leq n+i\}$ ($N\in\N$) where $e_0(X):=X_0^{2d}$, $X^\prime:=(X_1,\ldots,X_n)$ and $e_j^\prime(X^\prime):=e_j(1,X^\prime)$ for $j=1,\ldots,N$. We introduce the maps
	$$\begin{array}{rrccc} v&\colon&\C^{n+1}&\to&\C^{N+1} \\
							  &&x&\mapsto&(e_0(x),\ldots,e_N(x)) \\
							  v^\prime &\colon&\C^n&\to&\C^N \\
							  &&x^\prime&\mapsto&(e_1^\prime(x^\prime),\ldots,e_N^\prime(x^\prime))
	\end{array}$$
    and verify the assertion in three steps: \\

     \noindent \underline{Claim 1}: $\overline{\mathrm{conv}(v^\prime(\R^n))}$ is not a spectrahedral shadow \\
     \noindent Fix $f\in\Delta_{n+1,2d}\cap\mathfrak{S}_{n+i}$. Such a choice can be made by Corollary \ref{cor:specialf}. Let $L$ be the finite-dimensional linear space spanned by $\{e_1^\prime,\ldots,e_n^\prime\}$ in $\R[X^\prime]$ and $R\supsetneq\R$ a real closed field extension with canonical valuation ring $B$ and maximal ideal $\mathfrak{m}_B$. Set $\epsilon>0$ to be an infinitesimal element in $R$ and $\eta\in \R^n$. Then $f(\epsilon,X^\prime)\in B[X^\prime]$ is not SOS in $B[X^\prime]/\langle X_1,\ldots,X_n\rangle^{2d+1}B[X^\prime]$ by \cite[Proposition 4.18]{Scheiderer}. By \cite[Lemma 4.17]{Scheiderer}, $f_\eta(X^\prime):=f(\epsilon,X^\prime-\eta)$ is consequently not SOS in $B[\C^n]/(M_{\C^n,\eta})^{2d+1}$ where $M_{\C^n,\eta}\subseteq B[\C^n]$ is the kernel of the evaluation map $B[\C^n]\to B, g\mapsto g(\eta)$. 
 
     Observe that since $f$ is PSD on $\R^{n+1}$, $f(X_0,X_1-\eta_1,\ldots,X_n-\eta_n)$ is PSD on $\R^{n+1}$ and with that on $R^{n+1}$. Set $L_B:=L\otimes B$ which can be identified with the span of $\{e_1^\prime,\ldots,e_N^\prime\}$ in $B[X^\prime]$. If $f_\eta\in L_B+B1$, then $\overline{\mathrm{conv}(v^\prime(\R^n))}$ fails to be a spectrahedral shadow according to \cite[Proposition 4.19]{Scheiderer}. Thus, it suffices to show $f_\eta\in L_B+B1$: \\
    \noindent Since $i\geq k(n,d-1)-n+1$ (in both cases), all non-constant monomials in $X^\prime$ of degree at most $2d-2$ are among $e^\prime_1,\ldots,e^\prime_N$. Also, $f$ lies in the $\R$-span of $\{e_0,\ldots,e_N\}$. Hence, $f_\eta$ lies in the $\R$-span of $\{e_0(\epsilon,X^\prime-\eta),\ldots,e_N(\epsilon,X^\prime-\eta)\}$. Let $j=0,\ldots,N$ and distinguish three cases: \\
     \noindent\underline{Case 1}: If $X_0^2$ divides $e_j$, then $e_j(\epsilon,X^\prime-\eta)$ is a linear combination of monomials in $X^\prime$ of degree at most $2d-2$ with scalars in $B$ (coming from products of $\epsilon,\eta_1,\ldots,\eta_n\in B$). It follows $e_j(\epsilon,X^\prime-\eta)\in L_B+B1$. \\
    \noindent \underline{Case 2}: If $e_j(X)=X_0(X^\prime)^\beta$ for some $\beta\in\N_0^{n}$ with $\vert \beta \vert =2d-1$, then we conclude that $e_j(\epsilon,X^\prime-\eta)=\epsilon(X^\prime-\eta)^\beta$ is a linear combination of monomials in $X^\prime$ of degree at most $2d-2$ with scalars in $B$ (coming from products of $\epsilon,\eta_1,\ldots,\eta_n\in B$) and $(X^\prime)^\beta=e_j(1,X^\prime)=e_j^\prime(X^\prime)$ with scalar $\epsilon\in B$. Thus, $e_j(\epsilon,X^\prime-\eta)\in L_B+B1$. \\
    \noindent \underline{Case 3}: If $X_0$ does not divide $e_j$, then $e_j=m_sm_t$ for some $0\leq s,t\leq n+i$ with $\alpha_{s,0}=\alpha_{t,0}=0$. So, $e_j(\epsilon,X^\prime-\eta)=\epsilon(X^\prime-\eta)^\beta$ is a linear combination of monomials in $X^\prime$ of degree at most $2d-2$ with scalars in $B$ (coming from products of $\epsilon,\eta_1,\ldots,\eta_n\in B$), $e_j$ with scalar $\epsilon\in B$ and monomials in $X^\prime$ of degree at most $2d-1$ with scalars in $B$ (coming from poducts of $\epsilon,\eta_1,\ldots,\eta_n\in B$) that are of the type $(X^\prime)^{\alpha_s^\prime+\alpha^\prime_t-\chi_l}$ for appropriate $\alpha_s^\prime:=(\alpha_{s,1},\ldots,\alpha_{s,n})$, $\alpha_{t}^\prime:=(\alpha_{t,1},\ldots,\alpha_{t,n})$ and $1\leq l \leq n$ such that $(\alpha_s^\prime+\alpha_t^\prime)_l\geq 1$, where $\chi_l$ denotes the $l^{th}$ unit vector in $\Z^n$. W.l.o.g.\ assume $\alpha_{s,l}^\prime\geq 1$ and set $m(X):=X_0(X^\prime)^{\alpha_s^\prime-\chi_l}\geq_{lex}m_{n+i}(X)$. Also, $m_t\geq_{lex} m_{n+i}$ by choice and, so, $X_0(X^\prime)^{\alpha_s^\prime+\alpha^\prime_t-\chi_l}=m(X)m_t(X)\in\{e_1,\ldots,e_N\}$. Hence, we conclude that $(X^\prime)^{\alpha_s^\prime+\alpha^\prime_t-\chi_l}=m(1,X^\prime)m_t(1,X^\prime)\in\{e_1^\prime,\ldots,e_N^\prime\}$. Altogether, $e_j(\epsilon,X^\prime-\eta)\in L_B+B1$. \\
 
    \noindent \underline{Claim 2}: $\overline{\mathrm{conv}(v(\R^{n+1}))}$ is not a spectrahedral shadow: \\
    \noindent Assume for a contradiction that $\overline{\mathrm{conv}(v(\R^{n+1}))}$ is a spectrahedral shadow and consider the affine hyperplane $H:=\{z\in\R^{N+1}\mid z_0=1\}$. Since $H$ is also a spectrahedral shadow, $\overline{\mathrm{conv}(v(\R^{n+1}))}\cap H$ is a spectrahedral shadow. Furthermore, $\overline{\mathrm{conv}(v(\R^{n+1}))}\cap H=\overline{\mathrm{conv}(v(\{1\}\times \R^{n}))}$ and, thus, $\overline{\mathrm{conv}(v(\{1\}\times \R^{n}))}$ is a spectrahedral shadow. Consequently, it follows for the linear projection $\Pi\colon\R^{N+1}\to\R^N, (z_0,z^\prime)\mapsto z^\prime$ that $\Pi(\overline{\mathrm{conv}(v(\{1\}\times \R^{n}))})=\overline{\mathrm{conv}(v^\prime(\R^{n}))}$ is a spectrahedral shadow. A contradiction to the first (verified) claim. \\
    
    \noindent \underline{Claim 3}: $C_i$ is not a spectrahedral shadow: \\
    \noindent Assume for a contradiction that $C_i$ is a spectrahedral shadow. Recall from Theorem \ref{thm:suffCi} that $\mathcal{P}_{n+1,2d}\cap\mathfrak{S}_{n+i}=C_{i}\cap\mathfrak{S}_{n+i}$, which especially is a spectrahedral shadow as the intersection of two spectrahedral shadows. However, $\mathcal{P}_{n+1,2d}\cap\mathfrak{S}_{n+i}$ is isomorphic to the dual cone $(\overline{\mathrm{conv}(v(\R^{n+1}))})^\vee$. Hence, $(\overline{\mathrm{conv}(v(\R^{n+1}))})^\vee$, and with that also the bidual cone $\overline{\mathrm{conv}(v(\R^{n+1}))})$, is a spectrahedral shadow. A contradiction to the second (verified) claim.
    \end{proof}

\begin{thm} \label{thm:notssjump} Let $n,d\geq 1$ and $i=0,\ldots,k-n$. If $C_i$ is not a spectrahedral shadow as a subcone of $\mathcal{P}_{n+1,2d}$, then $C_i$ is not a spectrahedral shadow as a subcone of $\mathcal{P}_{n+1,2\delta}$ for $\delta\geq d$.
\end{thm}
    \begin{proof} Recall from Theorem \ref{cor:extendingCi} that $C_i=C_{\phi(K_i)}$ as a subcone of $\mathcal{P}_{n+1,2d}$ and $\mathcal{P}_{n+1,2\delta}$, respectively. For a proof by contraposition, we assume that $C_{\phi(K_i)}$ is a spectrahedral shadow as a subcone of $\mathcal{P}_{n+1,2\delta}$ and set $\mathfrak{S}$ to be the $\R$-span of $\{X^\alpha \mid \alpha\in\Z_0^{n+1} \colon \vert\alpha\vert = 2\delta \wedge \alpha_0\geq 2\delta-2d \}$. It follows that $C_{\phi(K_i)}\cap\mathfrak{S}$ is a spectrahedral shadow as the intersection of two spectrahedral shadows. Consequently, the image of $C_{\phi(K_i)}\cap\mathfrak{S}$ under the linear map $$\begin{array}{rrccc} \pi &\colon& \mathfrak{S} &\to & \mathcal{F}_{n+1,2d} \\
			&& f(X)=X_0^{2(\delta-d)}f^\prime(X) & \mapsto & f^\prime,
		\end{array}$$ that is $\pi(C_{\phi(K_i)}\cap\mathfrak{S})=C_{\phi(K_i)}$, is a spectrahedral shadow.  
  \end{proof}

\begin{proof}[\textbf{Proof of Theorem \ref{thm:spectrahedralshadows}}]
    We verify the assertion by an induction on $d$. The non-Hilbert cases to be considered as the inductive base cases are $(n+1,4)_{n\geq 3}$ and $(3,6)$. Theorem \ref{thm:notss} especially validates the assertion for these inductive base cases. Hence, we now assume that the assertion was already verified up to some $d$ and investigate the situation for $d+1$.

    By the inductive assumption, for $i=n+2,\ldots,k(n,d)-n$, the convex cone $C_{i}$ is not a spectrahedral shadow as a subcone of $\mathcal{P}_{n+1,2d}$. Thus, for $i=n+2,\ldots,k(n,d)-n$, the convex cone $C_{i}$ is not a spectrahedral shadow as a subcone of $\mathcal{P}_{n+1,2d+2}$ by Theorem \ref{thm:notssjump}. The same is true for $C_{n+1}$ when $n\geq 3$. It remains to show that, for $i=k(n,d)-n+1,\ldots,k(n,d+1)-n$, the convex cone $C_{i}$ is not a spectrahedral shadow as subcones of $\mathcal{P}_{n+1,2d+2}$. This, however, is true by Theorem \ref{thm:notss} (2). 
\end{proof}

\section*{Acknowledgements} 
  The first author acknowledges the support through \textit{Oberwolfach Leibniz Fellows} program. The second author is grateful for the support provided by the scholarship program of the University of Konstanz under the \textit{Landesgraduiertenfördergesetz} and the \textit{Studienstiftung des deutschen Volkes}.
  The third author acknowledges the support of \textit{Ausschuss für Forschungsfragen der Universität Konstanz}.
%%%%%%%%%%%%%%%%%%%%%%%%%%%%%%%%%%%%%%%%%%%%%%%%%%%%%%%%%%%%%%%%%%%%%%%%%%%%%%%%%%%%%%%%%%%%%%%%%%%%%%%%%%%%%%%%%%%%%%%%%%%%%%%%


\begin{thebibliography}{9}
	\bibitem{CKLR}
    M.\ D.\ Choi, M.\ Knebusch, T.\ Y.\ Lam, B.\ Reznick, Transversal zeros and positive semidefinite forms, in: \emph{Real algebraic geometry and quadratic forms}, Rennes, 1981, pp.\ 273 -- 298, \emph{Lecture Notes in Math.}, \textbf{959} Springer-Verlag, Berlin -- New York, 1982. 

    \bibitem{CL2}
	M.\ D.\ Choi and T.\ Y.\ Lam, Extremal positive semidefinite forms, \emph{Math.\ Ann.} \textbf{231} (1977/78), no.\ 1, 1 -- 18.

    \bibitem{CL}
	M.\ D.\ Choi and T.\ Y.\ Lam, An old question of Hilbert, in: \emph{Proceedings Queen's Univ.: Conference on Quadratic Forms}, Kingston, 1976, pp.\ 385 -- 405, \emph{Queen's Papers in Pure and Appl.\ Math.} \textbf{46} (1977), 385 -- 405.

    \bibitem{CLR}
	M.\ D.\ Choi, T.\ Y.\ Lam and B.\ Reznick, Sums of squares of real polynomials, in: \emph{$K$-theory and algebraic geometry: connections with quadratic forms and division algebras}, Santa Barbara, 1992, pp.\ 103 -- 126, \emph{Proc.\ Sympos.\ Pure Math.} \textbf{58} (1995), Part 2.

    \bibitem{GHK}
	C.\ Goel, S.\ Hess and S.\ Kuhlmann, Intermediate Cones between the Cones of Positive Semidefinite Forms and Sums of Squares, \textit{arXiv} 2303.13178.

    \bibitem{Hilbert}
	D.\ Hilbert, Über die Darstellung definiter Formen als Summe von Formenquadraten, \emph{Math.\ Ann.} \textbf{32} (1888), no.\ 3, 342 -- 350.

    \bibitem{HeltonNie}
    J.\ W.\ Helton and J.\ Nie, Sufficient and necessary conditions for semidefinite representability of convex hulls and sets, \emph{SIAM J.\ Optim.} \textbf{20} (2009), no.\ 2, 759 -- 791.

    \bibitem{HN2}
    J.\ W.\ Helton and J.\ Nie, Semidefinite representation of convex sets, \emph{Math. Program.} \textbf{122}, Ser.\ A (2010), no.\ 1, 21 -- 64.

    \bibitem{Motzkin}
	T.\ S.\ Motzkin, The arithmetic-geometric inequality, in: \emph{Proc.\ Sympos.\ Wright-Patterson Air Force Base: Inequalities}, Ohio, 1965, pp.\ 205 -- 224, Academic Press, New York, 1967.

    \bibitem{Mumford}
    D.\ Mumford, \emph{Algebraic geometry. I. Complex projective varieties}, Grundlehren der mathematischen Wissenschaften, 221, Springer-Verlag, Berlin -- New York, 1981.

    \bibitem{Nemirovski}
    A.\ Nemirovski, Advances in convex optimization: conic programming, in: \emph{International Congress of Mathematicians}, Vol.\ I, pp.\ 413 -- 444, \emph{European Mathematical Society (EMS)}, Zürich, 2007. 
 
    \bibitem{NetzerSanyal}
    T.\ Netzer and R.\ Sanyal, Smooth hyperbolicity cones are spectrahedral shadows, \emph{Math.\ Program.} \textbf{153}, Ser.\ B (2015), no.\ 1, 213 -- 221.

    \bibitem{NetzerPlaumann}
    T.\ Netzer and D.\ Plaumann, \emph{Geometry of linear matrix inequalities -- A course in convexity and real algebraic geometry with a view towards optimization}, Compact Textb.\ Math., Birkhäuser/Springer, Cham, 2023.

    \bibitem{Rez2}
    B.\ Reznick, Forms derived from the arithmetic-geometric inequality, \emph{Math.\ Ann.} \textbf{283} (1989), no.\ 3, 431 -- 464.

    \bibitem{Reznick}
    B.\ Reznick, Sums of even powers of real linear forms, \emph{Mem.\ Amer.\ Math.\ Soc.} \textbf{96} (1992), no.\ 463.

    \bibitem{Rob}
	R.\ M.\ Robinson, Some definite polynomials which are not sums of squares of real polynomials, in \emph{Selected questions of algebra and logic (a collection dedicated to the memory of A.\ I.\ Mal'cev)},  Izdat.\ "Nauka" Sibirsk.\ Otdel., Novosibirsk (1973), 264 -- 282. 

    \bibitem{Scheiderer2}
    C.\ Scheiderer, Semidefinite representation for convex hulls of real algebraic curves, \emph{SIAM J.\ Appl.\ Algebra Geom.} \textbf{2} (2018), no.\ 1, 1 -- 25.

    \bibitem{Scheiderer}
    C.\ Scheiderer, Spectrahedral shadows, \emph{SIAM J.\ Appl.\ Algebra Geom.} \textbf{2} (2018), no.\ 1, 26 -- 44.
\end{thebibliography}
\end{document}